\documentclass{amsart}
\usepackage{amssymb,latexsym,enumitem,graphics}
\usepackage{amsthm,amsmath,amssymb,amsfonts,mathrsfs}

 \newtheorem{thm}{Theorem}[section]
 \newtheorem{cor}[thm]{Corollary}
 \newtheorem{lem}[thm]{Lemma}
 
 \theoremstyle{definition}
 
 \theoremstyle{remark}

 \numberwithin{equation}{section}

\theoremstyle{remark}

\begin{document}
\title[PRODUCTS OF COMPOSITION AND DIFFERENTIATION OPERATORS ON THE HARDY SPACE]{PRODUCTS OF COMPOSITION AND DIFFERENTIATION OPERATORS}
\author[M. Moradi and M. Fatehi]{Mahbube Moradi and Mahsa Fatehi*}

\date{}
\address{Department of Mathematics, Shiraz Branch, Islamic Azad
University, Shiraz, Iran. }
\email{mathcall2021@yahoo.com}
\address{Department of Mathematics, Shiraz Branch, Islamic Azad
University, Shiraz, Iran.}
\email{fatehimahsa@yahoo.com}

\begin{abstract}
We consider products of composition and differentiation operators on the Hardy space. We provide a complete characterization of boundedness and compactness of these operators. Furthermore, we obtain the explicit condition for these operators to be  Hilbert-Schmidt operators.
\end{abstract}
\subjclass[2010]{47B38 (Primary), 30H10, 47E99}
\keywords{
Composition operator, differentiation operators, boundedness, compactness.\\
  *Corresponding author}                                            
\maketitle
\thispagestyle{empty}
\section{PRELIMINARIES}
Let  $\mathbb{D}$ be  the open unit disk in the complex plane $\Bbb{C}$.
The \textit{Hardy space}  $H^{2}$ is the Hilbert space of all analytic functions $f$ on $\mathbb{D}$ such that
\begin{align*}
\|f\|^{2}=\lim\limits_{r\rightarrow 1}\frac{1}{2\pi}\int_{0}^{2\pi}|f(re^{i\theta})|^{2}d\theta <\infty.
\end{align*}
It is well known that the Hardy space $H^{2}$ is a reproducing kernel Hilbert space, with the inner product \begin{align*}
\langle f,g\rangle=\frac{1}{2\pi}\int_{0}^{2\pi}f(e^{i\theta})\overline{g(e^{i\theta})}d\theta,
\end{align*}
and with
kernel functions $K_{w}^{(n)}(z)=\frac{n!z^{n}}{(1-\overline{w}z)^{n+1}}$, where $ n$ is a non-negative integer and $z,w \in \mathbb{D}$. These kernel functions satisfy $ \langle f, K_{w}^{(n)}\rangle=f^{(n)}(w)$ for each $f \in H^{2}$. To simplify notation we write $K_{w}$ in case $n=0$.
In particular note that  $\|K_{w}\|^{2}=K_{w}(w)=\frac{1}{1-|w|^{2}}$. Let $\hat{f}(n)$ be the $n$th coefficient of $f$ in its
Maclaurin series. Moreover, we have another representation for the norm
of $f$ on $H^{2}$ as follows
$$\|f\|^{2}=\sum_{n=0}^{\infty}|\hat{f}(n)|^{2}<\infty.$$
The space $H^{\infty}$ is the Banach space of bounded  analytic  functions $f$ on $\mathbb{D}$ with
$\|f\|_{\infty}=\sup \{|f(z)|:z\in \mathbb{D}\}$. \par
For $\varphi$ an analytic self-map of $\mathbb{D}$, the \textit{composition operator} $C_{\varphi}$  is defined for analytic functions $f$ on  $\mathbb{D}$ by $C_{\varphi}(f)= f\circ \varphi$. It is well known that every composition operator $C_{\varphi}$ is bounded on $H^{2}$ (see \cite[Corollary 3.7]{2}).
For each positive integer $k$, the operator $D^{(k)}$ for any $f \in H^{2}$ is defined by the rule $D^{(k)}(f)=f^{(k)}$. This operator is called the \textit{differentiation operator of order $k$}. For convenience, we use the notation $D$ when $k=1$.  The differentiation operators $D^{(k)}$ are unbounded on $H^{2}$, whereas Ohno \cite{13} found a characterization for $C_{\varphi}D$ and $DC_{\varphi}$ to be bounded and compact on $H^{2}$. The study of operators $C_{\varphi}D$ and $DC_{\varphi}$ was initially addressed by Hibschweiler,  Portnoy, and Ohno (see \cite{12} and \cite{13}) and has been noticed by many researchers (\cite{fh}, \cite{fh1}, and \cite{15}). In this paper, we will be considering a slightly broader class of these operators.
For each positive integer $n$, we write $D_{\varphi,n}$ to denote the operator on $H^{2}$ given by the rule $D_{\varphi,n}(f)=C_{\varphi}D^{(n)}(f)=f^{(n)}\circ\varphi$.
Our main results provide  complete characterizations of the boundedness and compactness of operators $D_{\varphi,n}$  on $H^{2}$   (Theorems \ref{thm1} and  \ref{thm2}).  In addition, we characterize the Hilbert-Schmidt operators  $D_{\varphi,n}$  on $H^{2}$ (Theorem \ref{thm3333}). In this paper, we use some ideas which are found  in \cite{13}.


   Let $\varphi$ be an analytic self-map of $\mathbb{D}$. The \textit{Nevanlinna counting function} $N_{\varphi}$ of $\varphi$ is defined by
   \begin{align*}
   N_{\varphi}(w)=\sum_{\varphi(z)=w}\log \big(1/|z|\big)\qquad w\in \mathbb{D} \setminus \{\varphi(0)\}
   \end{align*}
   and $N_{\varphi}(\varphi(0))=\infty$. Note that $ N_{\varphi}(w)=0$ when $w$ is not in $\varphi(\mathbb{D})$.
   For each $f \in H^{2}$, by using  change of variables formula and Littlewood-Paley Identity, the norm of $C_{\varphi}f$ is determined as follows:
   \begin{equation}\label{1-1}
   \|f\circ \varphi\|^{2}=\big| f\big(\varphi (0)\big)\big|^{2}+2\int_{\mathbb{D}}|f^{\prime}(w)|^{2}N_{\varphi}(w)dA(w),
   \end{equation}
   where $dA$ is the normalized  area measure on $\mathbb{D}$ (see \cite[Theorem 2.31]{2}).
Moreover, to obtain the lower bound estimate on $\|D_{\varphi,n}\|$ we need the following well known lemma as follows (see \cite[p. 137]{2}):\\ \par
Suppose that $\varphi$ is an analytic self-map of $\mathbb{D}$ and $f$ is analytic in $\mathbb{D}$. Assume that $\Delta$ is any disk not containing $\{f^{-1}(\varphi(0))\}$ and centered at $a$. Then
\begin{equation}\label{1-2}
N_{\varphi}(f(a))\leq \frac{1}{|\Delta|}\int_{\Delta}N_{\varphi}(f(w))dA(w),
 \end{equation}
where $|\Delta|$ is the normalized area measure of $\Delta$.



\section{Boundedness  and   compactness of $D_{\varphi,n}$}
The goal of this section is to determine which of these operators $D_{\varphi,n}$ are bounded and compact.

\begin{thm}\label{thm1}
Let  $\varphi$ be an analytic self-map of $\mathbb{D}$ and $n$ be a positive integer.  The operator $D_{\varphi , n}$ is bounded on  $H^{2}$  if and only if
\begin{align*}
N_{\varphi}(w)=O\bigg(\bigg[\log \big(1/ |w|\big)\bigg]^{2n+1}\bigg)\qquad \big(|w|\rightarrow 1\big).
\end{align*}
\end{thm}
\begin{proof}
Suppose that $D_{\varphi,n}$ is bounded on $H^{2}$. Let $f(z)=\frac{K_{\lambda}(z)}{\|K_{\lambda}\|}=\frac{\sqrt{1-|\lambda|^{2}}}{1-\overline{\lambda}z}$ for $\lambda\in\mathbb{D}$. By  (\ref{1-1}), we see that
\begin{align}\label{2-1}
\|D_{\varphi ,n}\|^{2}&\geq \| D_{\varphi ,n}f\|^{2}\notag\\
&=\bigg\| C_{\varphi}\bigg(\frac{n!\overline{\lambda}^{n}\sqrt{1-|\lambda |^{2}}}{\big(1-\overline{\lambda}z\big)^{n+1}}\bigg)\bigg\|^{2}\notag\\
&=\bigg |\frac{n!\overline{\lambda}^{n}\sqrt{1-|\lambda|^{2}}}{\big(1-\overline{\lambda} \varphi(0)\big)^{n+1}}\bigg |^{2}+2\int_{\mathbb{D}} \bigg |\frac{(n+1)!\overline{\lambda}^{n+1}\sqrt{1-|\lambda |^{2}}}{\big(1-\overline{\lambda}w\big)^{n+2}}\bigg |^{2}N_{\varphi}(w)dA(w)\notag\\
&\geq \int_{\mathbb{D}}\frac{2\big((n+1)!\big)^{2}|\lambda |^{2n+2}\big(1-|\lambda |^{2}\big)}{\big|1-\overline{\lambda}w\big|^{2n+4}}N_{\varphi}(w)dA(w).
\end{align}
Substitute
$w=\alpha_{\lambda}(u)=\frac{\lambda-u} {1-\overline{\lambda}u}$ back into  (\ref{2-1}) and using \cite[Theorem 7.26]{726} to obtain
\begin{align}\label{2-2}
\|D_{\varphi,n}\|^{2}&\geq \int_{\mathbb{D}}\frac{2\big((n+1)!\big)^{2}|\lambda |^{2n+2}\big(1-|\lambda |^{2}\big)}{\big|1-\overline{\lambda}\alpha_{\lambda}(u)\big|^{2n+4}}N_{\varphi}(\alpha_{\lambda}(u))\big|\alpha_{\lambda}^{\prime}(u)\big|^{2}dA(u).
\end{align}
Since $1-\overline{\lambda}\alpha_{\lambda}(u)=\frac{1-|\lambda|^{2}}{1-\overline{\lambda}u}$
  and
 $\alpha_{\lambda}^{\prime}(u)=\frac{|\lambda|^{2}-1}{\big(1-\overline{\lambda}u\big)^{2}}$, by substituting $\alpha_{\lambda}^{\prime}$ and $1-\overline{\lambda}\alpha_{\lambda}$ back into  (\ref{2-2}), we see that
\begin{align}\label{2-3}
\|D_{\varphi,n}\|^{2}&\geq\int_{\mathbb{D}}\frac{2\big((n+1)!\big)^{2}|\lambda |^{2n+2}|1-\overline{\lambda}u|^{2n}}{\big(1-|\lambda |^{2}\big)^{2n+1}}N_{\varphi}(\alpha_{\lambda}(u))dA(u).
\end{align}
Because $\big|1-\overline{\lambda}u\big|\geq\frac{1}{2}$ for any $u\in{\mathbb{D}}/{2}$, we get from  (\ref{2-3}) that
\begin{align}
\|D_{\varphi,n}\|^{2}\geq \int_{{\mathbb{D}}/{2}}\frac{2\big((n+1)!\big)^{2}|\lambda |^{2n+2}}{2^{2n}\big(1-|\lambda |^{2}\big)^{2n+1}}N_{\varphi}(\alpha_{\lambda}(u))dA(u).\label{eq1}
\end{align}
There exists $r<1$ such that for each $\lambda$ with $r<|\lambda|<1$, $\alpha_{\lambda}^{-1}(\varphi(0))\notin {\mathbb{D}}/{2}$ because $|\alpha_{\lambda}^{-1}(\varphi(0))|=|\alpha_{\varphi(0)}(\lambda)|$ and $\alpha_{\varphi(0)}$ is an automorphism of $\mathbb{D}$. By (\ref{1-2}) and  (\ref{eq1}), we have
 \begin{align}
 \|D_{\varphi,n}\|^{2}&\geq  \frac{2\big((n+1)!\big)^{2}|\lambda |^{2n+2}}{2^{2n}\big(1-|\lambda|^{2}\big)^{2n+1}}\int_{{\mathbb{D}}/{2}}N_{\varphi}(\alpha_{\lambda}(u))dA(u)\notag\\
&\geq \frac{2\big((n+1)!\big)^{2}|\lambda |^{2n+2}}{2^{2n}\big(1-|\lambda|^{2}\big)^{2n+1}}\cdot\frac{N_{\varphi}(\alpha_{\lambda}(0))}{4}\notag\\
&=\frac{\big((n+1)!\big)^{2}|\lambda |^{2n+2}}{2^{2n+1}\big(1-|\lambda|^{2}\big)^{2n+1}}N_{\varphi}(\lambda)\label{2.22}
 \end{align}
for each $\lambda$ with $r<|\lambda|<1$.
Since $D_{\varphi ,n}$ is bounded, there exists  a constant number $M$ so that
\begin{align}\label{eq2-5*}
\lim\limits_{|\lambda|\rightarrow1}\frac{\big((n+1)!\big)^{2}|\lambda|^{2n+2}}{2^{2n+1}\big(1-|\lambda|^{2}\big)^{2n+1}}N_{\varphi}(\lambda)\leq M.
\end{align}
We know that $\log\big({1}/{|\lambda|}\big)$ is comparable to $1-|\lambda|$ as  $|\lambda| \rightarrow 1^{-}$.
 Note that
 \begin{align}
& \lim\limits_{|\lambda |\rightarrow 1}\frac{\big((n+1)!\big)^{2}|\lambda|^{2n+2}}{2^{2n+1}\big(1-|\lambda|^{2}\big)^{2n+1}}N_{\varphi}(\lambda)\notag\\
&= \lim\limits_{|\lambda |\rightarrow 1}\frac{\big((n+1)!\big)^{2}|\lambda |^{2n+2}}{2^{2n+1}\big(1+|\lambda |\big)^{2n+1}}\bigg(\frac{\log \big({1}/{|\lambda |}\big)}{1-|\lambda |}\bigg)^{2n+1}\frac{N_{\varphi}(\lambda)}{\big(\log \big({1}/{|\lambda|}\big)\big)^{2n+1}}\notag\\
& \geq \frac{\big((n+1)!\big)^{2}}{2^{6n+4}}\lim\limits_{|\lambda |\rightarrow 1}\frac{N_{\varphi}(\lambda)}{\big(\log\big({1}/{|\lambda|}\big)\big)^{2n+1}}.\label{2-6}
 \end{align}
 By \eqref{eq2-5*} and  \eqref{2-6}, we can see that
 $$N_{\varphi}(\lambda)=O\bigg(\bigg[\log\big({1}/{|\lambda|}\big)\bigg]^{2n+1}\bigg)\qquad (|\lambda|\rightarrow 1).$$

Conversely, suppose that for  some  $R$ with $0<R<1$, there  exists a constant $M$ satisfying
\begin{align*}
\sup_{R<|w|<1}N_{\varphi}(w)/\bigg [\log\big({1}/{|w|}\big)\bigg]^{2n+1}\leq M.
\end{align*}
Let $f$ be an arbitrary function in $H^2$. It follows from (\ref{1-1}) that
\begin{align}
\| D_{\varphi ,n}f\| ^{2}&=\big |f^{(n)}\big(\varphi (0)\big)\big |^{2}+2\int_{\mathbb{D}}\big |f^{(n+1)}(w)\big |^{2}N_{\varphi}(w)dA(w)\notag\\
&=\big |f^{(n)}\big(\varphi (0)\big)\big |^{2}\notag\\
&+2\bigg(\int_{R\mathbb{D}}\big|f^{(n+1)}(w)\big |^{2}N_{\varphi}(w)dA(w)+\int_{\mathbb{D}\backslash R\mathbb{D}}\big| f^{(n+1)}(w)\big |^{2}N_{\varphi}(w)dA(w)\bigg).\label{2.2}
\end{align}
First we estimate the  first and the second terms  in the right-hand of (\ref{2.2}).  Observe that
\begin{align*}
f^{(n)}(z)=\big\langle f,K_{z}^{(n)}\big\rangle=\int_{0}^{2\pi}\frac{n!e^{-in\theta}f(e^{i\theta})}{\big(1-e^{-i\theta}z\big)^{n+1}}\frac{d\theta}{2\pi}
\end{align*}
and  hence
\begin{align}\label{222}
&\big|f^{(n)}(z)\big|\leq \frac{n!}{\big(1-|z|\big)^{n+1}}\int_{0}^{2\pi}\big|f(e^{i\theta})\big|\frac{d\theta}{2\pi}\leq \frac{n!}{\big(1-|z|\big)^{n+1}}\|f\|
\end{align}
for any $z\in \mathbb{D}$. It follows  from (\ref{222}) that
\begin{align}\label{eq2.3}
 \big|f^{(n)}\big(\varphi (0)\big)\big|\leq \frac{n!\|f\|}{\big(1-|\varphi(0)|\big)^{n+1}}.
\end{align}
Moreover,  we can see that
\begin{align}\label{eq2}
\big|f^{(n+1)}(z)\big|=\big|\big\langle f,K_{z}^{(n+1)}\big\rangle\big|=\frac{(n+1)!}{\big(1-|z|\big)^{n+2}}\|f\|
\end{align}
for any  $z\in\mathbb{D}$.
Therefore by  \eqref{eq2}, we see that
\begin{align*}
\int_{R\mathbb{D}}\big |f^{(n+1)}(w)\big |^{2}N_{\varphi}(w)dA(w)&\leq \bigg(\frac{(n+1)!}{(1-R)^{n+2}}\bigg)^{2}\|f\|^{2}\int_{R\mathbb{D}}N_{\varphi}(w)dA(w).
\end{align*}
Since $\|\varphi\|^{2}=|\varphi(0)|^{2}+2\int_{\mathbb{D}} N_{\varphi}(w)dA(w)$ by (\ref{1-1}),
we obtain
\begin{align}\label{2.11}
\int_{\mathbb{D}} N_{\varphi}(w)dA(w)=\frac{1}{2}\big(\|\varphi\|^{2}-|\varphi(0)|^{2}\big)<1.
 \end{align}
 From (\ref{eq2}) and \eqref{2.11}, we see that
\begin{align}\label{2.5}
\int_{R\mathbb{D}}\big |f^{(n+1)}(w)\big |^{2}N_{\varphi}(w)dA(w) &\leq  \bigg(\frac{(n+1)!}{(1-R)^{n+2}}\bigg)^{2}\|f\|^{2}.
\end{align}
Now we estimate the third term in the right-hand of (\ref{2.2}). We have
\begin{align}\label{eq3}
&\int_{\mathbb{D}\setminus R\mathbb{D}}\big |f^{(n+1)}(w)\big |^{2}N_{\varphi}(w)dA(w)\notag \\
&=\int_{\mathbb{D}\setminus R\mathbb{D}}\big |f^{(n+1)}(w)\big |^{2}\big(\log({1}/{|w|})\big)^{2n+1}\frac{N_{\varphi}(w)}{\big(\log({1}/{|w|})\big)^{2n+1}}dA(w)\notag\\
&\leq \sup_{R<|w|<1}\frac{N_{\varphi}(w)}{\big(\log({1}/{|w|})\big)^{2n+1}}\int_{\mathbb{D}\setminus R\mathbb{D}}\big |f^{(n+1)}(w)\big |^{2}\big(\log({1}/{|w|})\big)^{2n+1}dA(w)\notag\\
&\leq M\int_{\mathbb{D}\setminus R\mathbb{D}}\big |f^{(n+1)}(w)\big |^{2}\big(\log(1/|w|)\big)^{2n+1}dA(w).
\end{align}
Let  $f(z)=\sum_{m=0}^{\infty}a_{m}z^{m}$. 
  We get
\begin{align}\label{4}
&\int_{\mathbb{D}\setminus R\mathbb{D}}\big |f^{(n+1)}(w)\big |^{2}\big(\log ({1}/{|w|})\big)^{2n+1}dA(w) \notag \\
&\leq\int_{\mathbb{D}\setminus R\mathbb{D}}\bigg |\sum_{m=n+1}^{\infty}m(m-1)...(m-n)a_{m}(w)^{m-(n+1)}\bigg|^{2}\big(\log ({1}/{|w|})\big)^{2n+1}dA(w)\notag\\
&\leq \sum_{m=n+1}^{\infty}m^{2}(m-1)^{2}...(m-n)^{2}|a_{m}|^{2}\int_{\mathbb{D}\setminus R\mathbb{D}}\bigg |(w)^{m-(n+1)}\bigg |^{2}\big(\log ({1}/{|w|})\big)^{2n+1}dA(w)\notag\\
&\leq \sum_{m=n+1}^{\infty}m^{2}(m-1)^{2}...(m-n)^{2}|a_{m}|^{2}\int_{\mathbb{D}}\bigg |(w)^{m-(n+1)}\bigg |^{2}\big(\log ({1}/{|w|})\big)^{2n+1}dA(w)\notag\\
&=\sum_{m=n+1}^{\infty}m^{2}(m-1)^{2}...(m-n)^{2}|a_{m}|^{2}\int_{0}^{1}\int_{0}^{2\pi}|re^{i\theta}|^{2(m-(n+1))}\big(\log({1}/{r})\big)^{2n+1}rdr\frac{d\theta}{\pi}\notag\\
&\leq\sum_{m=n+1}^{\infty}m^{2}(m-1)^{2}...(m-n)^{2}|a_{m}|^{2}\int_{0}^{1}(r)^{2(m-(n+1))}\big(\log({1}/{r})\big)^{2n+1}2rdr.\notag\\
\end{align}
Now substitute $t=r^{2}$ and $u=\log ({1}/{t})$  to obtain  
\begin{align}\label{eq2.152}
\int_{0}^{1}(r)^{2(m-(n+1))}\big(\log({1}/{r})\big)^{2n+1}2rdr &=\int_{0}^{1}t^{(m-(n+1))}\bigg(\frac{1}{2}\log ({1}/{t})\bigg)^{2n+1}dt\notag\\
& =({1}/{2})^{2n+1}\int_{0}^{\infty}e^{-u(m-n)}u^{2n+1}du.
\end{align}
By substituting $x=(m-n)u$ back into \eqref{eq2.152}, we have
\begin{align}
({1}/{2})^{2n+1}\int_{0}^{\infty}e^{-u(m-n)}u^{2n+1}du &=\frac{1}{2^{2n+1}(m-n)^{2n+2}}\int_{0}^{\infty}e^{-x} x^{2n+1}dx\notag\\
&=\frac{\Gamma(2n+2)}{2^{2n+1}(m-n)^{2n+2}}.\label{eq5}
\end{align}
By  \eqref{eq3}, \eqref{4}, \eqref{eq2.152} and \eqref{eq5}, we can see that

\begin{align}
\int_{\mathbb{D}\setminus R\mathbb{D}}\big |f^{(n+1)}(w)\big |^{2}N_{\varphi}(w)dA(w) &\leq M\sum_{m=n+1}^{\infty}m^{2}(m-1)^{2}...(m-n)^{2}|a_{m}|^{2}\frac{\Gamma (2n+2)}{2^{2n+1}(m-n)^{2n+2}}\notag\\
&=M\frac{(2n+1)!}{2^{2n+1}}\sum_{m=n+1}^{\infty}\frac{m^{2}(m-1)^{2}...(m-n+1)^{2}}{(m-n)^{2n}}|a_{m}|^{2}\notag\\
&\leq M\lambda\frac{(2n+1)!}{2^{2n+1}}\sum_{m=n+1}^{\infty}|a_{m}|^{2}\notag\\
&\leq M\lambda\frac{(2n+1)!}{2^{2n+1}}\|f\|^{2},\label{eq6}
\end{align}
where $\lambda$ is a constant  so that $\frac{m^{2}(m-1)^{2}...(m-n+1)^{2}}{(m-n)^{2n}}\leq \lambda$ for each $m\geq n+1$
(note that the function  $f(x)=\frac{x^{2}(x-1)^{2}...(x-n+1)^{2}}{(x-n)^{2n}}$ is  bounded on $[n+1,+\infty)$).
Then  \eqref{2.2}, \eqref{eq2.3}, \eqref{2.5} and \eqref{eq6} show that  
  $D_{\varphi,n}$ is bounded.
\end{proof}

\begin{thm}\label{thm2}
Let $\varphi$ be an analytic self-map  of $\mathbb{D}$ and  $n$ be a positive   integer.  The operator  $D_{\varphi, n}$  is compact  on  $H^{2}$ if  and only if
\begin{align}\label{eq19}
N_{\varphi}(w)=o\bigg(\bigg[\log\big(1/|w|\big)\bigg]^{2n+1}\bigg)\qquad (|w|\rightarrow 1).
\end{align}
\end{thm}
\begin{proof}
Let $h_{m}(z)=\frac{\sqrt{1-|\lambda_{m}|^{2}}}{1-\overline{\lambda}_{m}z}$  for a  sequence $\{\lambda_{m}\}$ in $\mathbb{D}$   so that  $|\lambda_{m}|\rightarrow 1$ as $m\rightarrow \infty$.  Then $h_{m}\rightarrow 0$ weakly as  $m\rightarrow \infty$  by \cite[Theorem 2.17]{2}. First suppose that $D_{\varphi ,n}$ is  compact. Hence $\|D_{\varphi , n}h_{m}\|\rightarrow 0$ as $m\rightarrow \infty$.  Therefore (\ref{2.22})  shows that
\begin{align*}
\lim\limits_{m\rightarrow \infty}\frac{\big((n+1)!\big)^{2}|\lambda_{m}|^{2n+2}}{2^{2n+1}(1-|\lambda_{m}|^{2})^{2n+1}}N_{\varphi}(\lambda_{m})=0.
\end{align*}
Since $\log(1/ |\lambda_{m}|)$ is comparable to $1-|\lambda_{m}|$ as $ m\rightarrow \infty$, the result follows.

Conversely,  suppose that  \eqref{eq19} holds. Let $\epsilon>0$. Then there exists $R, 0<R<1$, such that
\begin{align}\label{2020}
\underset{R<|w|<1}{\sup}N_{\varphi}(w)/\big[\log(1/|w|)\big]^{2n+1}<\epsilon.
\end{align}
Let $\{f_{m}\}$ be any  bounded sequence  in $H^{2}$. By  using  the idea which was stated  in  the proof of \cite[Proposition 3.11]{2}, we can see  that $\{f_{m}\}$ is a  normal family and there  exists  a subsequence $\{f_{m_{k}}\}$ which  converges
 to  some function $f\in  H^{2}$
 uniformly  on all  compact subsets  of $\mathbb{D}$.  Let $g_{m_{k}}=f_{m_{k}}-f$ for each positive integer $k$.  Note that $\{g_{m_{k}}\}$ is a bounded sequence in $H^{2}$ which   converges to $0$ uniformly on all compact subsets  of $\mathbb{D}$.
By \eqref{2.2}, we  obtain
\begin{align}
\|D_{\varphi ,n}g_{m_{k}}\|^{2}&=\big |g_{m_{k}}^{(n)}(\varphi(0))\big|^{2}+2\int_{R\mathbb{D}}\big |g_{m_{k}}^{(n+1)}(w)\big |^{2}N_{\varphi}(w)dA(w)\notag\\
&+2\int_{\mathbb{D}\setminus R\mathbb{D}}\big |g_{m_{k}}^{(n+1)}(w)\big|^{2}N_{\varphi}(w)dA(w).\label{eq2.20}
\end{align}
By \cite[Theorem 2.1, p. 151]{34}, 
 we can choose $k_{\epsilon}$ so that
 \begin{align}
 \big |g_{m_{k}}^{(n)}\big(\varphi(0)\big)\big|<\sqrt{\epsilon}\label{eq2.200}
  \end{align}
  and  $\big |g_{m_{k}}^{(n+1)}\big|<\sqrt{\epsilon}$ on $R\mathbb{D}$
 whenever $k>k_{\epsilon}$. Substituting $f(z)=z$ into  \eqref{1-1}, we see that
\begin{align}\label{eq2.22}
\int_{R\mathbb{D}}\big |g_{m_{k}}^{(n+1)}(w)\big |^{2}N_{\varphi}(w)dA(w)&\leq \epsilon \int_{R\mathbb{D}} N_{\varphi}(w)dA(w)\notag\\
&\leq \frac{\epsilon}{2}\big(\|\varphi \|^{2}-|\varphi(0)|^{2}\big)
\end{align}
for $k>k_{\epsilon}$.
On  the other hand by \eqref{2020} and the same idea as stated in the proof of \eqref{eq3} and \eqref{eq6},  we see that
\begin{align}
&\int_{\mathbb{D}\setminus R\mathbb{D}}\big |g_{m_{k}}^{(n+1)}(w)\big |^{2}N_{\varphi}(w)dA(w)\notag\\
&\leq \underset{R<|w|<1}{\sup}\frac{N_{\varphi}(w)}{\big[\log (1/|w|)\big]^{2n+1}}\int_{\mathbb{D}\setminus R\mathbb{D}}\big |g_{m_{k}}^{(n+1)}(w)\big |^{2}\big[\log (1/|w|)\big]^{2n+1}dA(w)\notag\\
&\leq C \epsilon \|g_{m_{k}}\|,\label{eq2.23}
\end{align}
where $C$ is a constant.
Hence we conclude that $\|D_{\varphi ,n}g_{m_{k}}\|$ converges to zero as $k\rightarrow \infty$ by  \eqref{eq2.20}, \eqref{eq2.200}, \eqref{eq2.22} and \eqref{eq2.23}  and so $D_{\varphi,n}$ is compact.
\end{proof}

The preceding theorems lead to characterizations of all  bounded and compact operators $D_{\varphi ,n}$ when $\varphi$  is a univalent self-map.

\begin{cor}
Let $\varphi$ be a univalent   self-map of $\mathbb{D}$  and $n$ be a positive integer. Then the following hold.
\begin{itemize}
\item[(i)]
$D_{\varphi ,n}$ is bounded  on $H^{2}$ if and only if
\begin{align*}
\sup_{w\in \mathbb{D}}\frac{1-|w|}{\big(1-|\varphi(w)|\big)^{2n+1}}<\infty
\end{align*}
\item[(ii)]
$D_{\varphi ,n}$ is compact on $H^{2}$ if  and  only if
\begin{align*}
\lim\limits_{|w|\rightarrow 1}\frac{1-|w|}{\big(1-|\varphi(w)|\big)^{2n+1}}=0.
\end{align*}
\end{itemize}
\end{cor}
\begin{proof}
Since  $\varphi$ is univalent, we can see that
$N_{\varphi}(w)=\log \big(1/|z|\big)$,
where $\varphi(z)=w$. We observe  that
\[\frac{N_{\varphi}(w)}{\big[\log(1/|w|)\big]^{2n+1}}=\frac{-\log \big(|z|\big)}{\big(-\log \big(|\varphi(z)|\big)\big)^{2n+1}}.\]
Moreover, we know that $\log\big(1/|z|\big)$ is comparable to $1-|z|$ as $|z|\rightarrow 1^{-}$. Furthermore $|z|\rightarrow 1$ as $|\varphi(z)|\rightarrow 1$. Therefore the results  follow immediately from Theorems \ref{thm1} and \ref{thm2}.
\end{proof}

\section{Hilbert-Schmidt operator $D_{\varphi,n}$}
We begin with a few easy observations that help us in the proof of Theorem \ref{thm3333}. In the proof of the following lemma, we assume that $0^{0}=1$.

\begin{lem}\label{lm1}
Let  $n$ be a positive integer and $\alpha_{k}>0$  for each $0\leq k\leq n$. Then  for $0\leq x< 1$, the following statements hold.

(a)    $ \sum_{k=0}^{n}\frac{\alpha_{k}x^{k}}{(1-x)^{n+k+1}}\leq \frac{\sum_{k=0}^{n}\alpha_{k}}{(1-x)^{2n+1}}.$

(b)  There exists a positive number $\beta$  such that
$\sum_{k=0}^{n}\frac{\alpha_{k}x^{k}}{(1-x)^{n+k+1}} \geq\frac{\beta}{(1-x)^{2n+1}} $.

\end{lem}
\begin{proof}

(a) We  can see that
\begin{align*}
\sum_{k=0}^{n}\frac{\alpha_{k}x^{k}}{(1-x)^{n+k+1}}=\frac{\sum_{k=0}^{n}\alpha_{k}x^{k}(1-x)^{n-k}}{(1-x)^{2n+1}}.
\end{align*}
Since $0\leq x<1$ and $\alpha_{k}>0$, we conclude  that
$ \sum_{k=0}^{n}\alpha_{k}x^{k}(1-x)^{n-k}\leq \sum_{k=0}^{n}\alpha_{k}.$ Hence the conclusion follows.

(b) We have
\begin{align*}
(1-x)^{2n+1}\sum_{k=0}^{n}\frac{\alpha_{k}x^{k}}{(1-x)^{n+k+1}}=\sum_{k=0}^{n}\alpha_{k}x^{k}(1-x)^{n-k}>0.
\end{align*}
Since $\sum_{k=0}^{n}\alpha_{k}x^{k}(1-x)^{n-k}$ is  a continuous function on  $[0,1]$, there exists  a positive number $\beta$ such  that  $\sum_{k=0}^{n}\alpha_{k}x^{k}(1-x)^{n-k}\geq \beta$.
Hence the result follows.

\end{proof}
\begin{lem}\label{lm3}
Let $n$ be a positive  integer. Then
\begin{align*}
\sum_{m=n}^{\infty}\big[m(m-1)...(m-n+1)\big]^{2}x^{m-n}
&=\big(n!\big)^{2}\sum_{k=0}^{n}\frac{(n+k)!}{\big(k!\big)^{2}(n-k)!}\frac{x^{k}}{(1-x)^{n+k+1}}
\end{align*}
for $0\leq x<1$.
\end{lem}

\begin{proof}
See \cite[Lemma 1]{15} and   the general Leibniz rule.
\end{proof}

 A \textit{Hilbert–Schmidt operator} on a separable  Hilbert space $H$ is a bounded operator $A$  with finite \textit{Hilbert–Schmidt norm}
    $\|A\|_{HS}=\left(\sum_{n=1}^{\infty}\|A e_{n}\|^{2}\right)^{1/2}$,
where $\{e_{n }\}$ is an orthonormal basis of $H$.  These definitions are independent of the choice of the basis (see \cite[Theorem 3.23]{2}).

\begin{thm} \label{thm3333}
Let $D_{\varphi,n}$ be  a bounded  operator on $H^{2}$. Then $D_{\varphi,n}$ is a Hilbert-Schmidt  operator on $H^{2}$ if and only if
\begin{align}\label{3.I}
\lim\limits_{r\rightarrow 1}\frac{1}{2\pi}\int_{0}^{2\pi}\frac{1}{\big(1-\big|\varphi(re^{i\theta})\big|^{2}\big)^{2n+1}}<\infty .
\end{align}
\end{thm}
\begin{proof}
Suppose  that \eqref{3.I} holds. Lemmas \ref{lm1}, \ref{lm3} and \cite[Theorem 1.27]{726} imply that
\begin{align}\label{3.II}
\sum_{m=0}^{\infty}\big\|D_{\varphi,n}z^{m}\big\| &=\sum_{m=n}^{\infty}\big\| m(m-1)...(m-n+1)\varphi^{m-n}\big\|\notag\\
&=\sum_{m=n}^{\infty}\lim\limits_{r\rightarrow 1}\frac{1}{2\pi}\int_{0}^{2\pi}\big|m(m-1)...(m-n+1)\varphi^{m-n}(re^{i\theta})\big|^{2}d\theta\notag\\
&=\lim\limits_{r\rightarrow 1}\sum_{m=n}^{\infty}\frac{1}{2\pi}\int_{0}^{2\pi}\big|m(m-1)...(m-n+1)\varphi^{m-n}(re^{i\theta})\big|^{2}d\theta\notag\\
&=\lim\limits_{r\rightarrow 1}\frac{1}{2\pi}\int_{0}^{2\pi}\sum_{m=n}^{\infty}\big|m(m-1)...(m-n+1)\varphi^{m-n}(re^{i\theta})\big|^{2}d\theta\notag\\
&=\lim\limits_{r\rightarrow 1}\frac{1}{2\pi}\int_{0}^{2\pi}\sum_{k=0}^{n} \frac{\big(n!\big)^{2}(n+k)!}{\big(k!\big)^{2}(n-k)!}\frac{\big|\varphi (re^{i\theta})\big |^{2k}}{\big(1-\big|\varphi(re^{i\theta})\big|^{2}\big)^{n+k+1}} \notag\\
&\leq \lim\limits_{r\rightarrow 1}\frac{1}{2\pi}\int_{0}^{2\pi}\frac{\alpha}{\big(1-\big|\varphi(re^{i\theta})\big|^{2}\big)^{2n+1}},
\end{align}
where $\alpha=\sum_{k=0}^{n}\frac{(n!)^{2}(n+k)!}{(k!)^{2}(n-k)!}$
(note that the interchange of limit and summation  is  justified by \cite[Corollary 2.23]{2}  and  using  Lebesgue's Monotone  Convergence  Theorem with counting  measure).
It follows  that $\sum_{m=0}^{\infty}\big\| D_{\varphi,n}z^{m}\big\|<\infty$ and  so  $D_{\varphi,n}$ is a Hilbert-Schmidt  operator on $H^{2}$ by \cite[Theorem 3.23]{2}.

Conversely, suppose that $D_{\varphi,n}$ is a Hilbert-Schmidt operator  on $H^{2}$. We infer from  \cite[Theorem 3.23]{2} that
\begin{align}\label{3.}
 \sum_{m=0}^{\infty}\big\|D_{\varphi,n}z^{m}\big\|^{2}<\infty.
\end{align}
On  the other hand, by   the proof of (3.2) and Lemma  \ref{lm1},  there exists a positive  number $\beta$  such that

\begin{align}\label{3.*}
\sum_{m=0}^{\infty}\big\|D_{\varphi,n}z^{m}\big\|^{2}&=\lim\limits_{r\rightarrow 1}\frac{1}{2\pi}\int_{0}^{2\pi}\sum_{k=0}^{n}\frac{(n!)^{2}(n+k)!}{(k!)^{2}(n-k)!}\frac{\big|\varphi (re^{i\theta})\big|^{2k}}{\big(1-\big|\varphi(re^{i\theta})\big|^{2}\big)^{n+k+1}}\notag\\
&\geq \lim\limits_{r\rightarrow 1} \frac{1}{2\pi}\int_{0}^{2\pi}\frac{\beta}{\big(1-\big|\varphi(re^{i\theta})\big|^{2}\big)^{2n+1}}.
\end{align}
Hence  the result follows  from  \eqref{3.} and \eqref{3.*}.
\end{proof}


 \end{document}